\documentclass[11pt]{article}

\usepackage{amsmath,amsfonts,amssymb,url,amsthm}
\usepackage{amscd}

\usepackage[utf8]{inputenc}

\usepackage[margin=3cm]{geometry}

\newtheoremstyle{yuri}
{} % space above
{} % space below
{\slshape} % body font
{} % indent amount
{\bfseries} % Theorem head font
{.} % ponctuation
{.5em} % space after theorem head
{} % Theorem head spec

\theoremstyle{yuri}
\newtheorem{proposition}{Proposition}[section]
\newtheorem{lemma}[proposition]{Lemma}
\newtheorem{corollary}[proposition]{Corollary}
\newtheorem{conjecture}[proposition]{Conjecture}
\newtheorem{theorem}[proposition]{Theorem}

\theoremstyle{definition}
\newtheorem{remark}[proposition]{Remark}

\newtheorem{definition}[proposition]{Definition}

\newcommand{\eps}{\varepsilon}

\newcommand{\C}{{\mathbb C}}

\renewcommand\P{{\mathbb P}}
\newcommand\Q{{\mathbb Q}}
\newcommand\R{{\mathbb R}}
\newcommand\Z{{\mathbb Z}}
\newcommand\PP{{\mathbb P}}

\newcommand{\CC}{{\mathcal{C}}}
\newcommand{\DD}{{\mathcal{D}}}

\newcommand{\LL}{{\mathcal{L}}}
\newcommand{\norm}{{\mathcal{N}}}
\newcommand{\OO}{{\mathcal{O}}}
\newcommand{\UU}{{\mathcal{U}}}

\newcommand{\tilP}{{\widetilde P}}
\newcommand{\tilCC}{{\widetilde \CC}}
\newcommand{\tilt}{{\tilde t}}

\newcommand{\et}{{\mathrm{et}}}
\newcommand{\fl}{{\mathrm{fl}}}
\newcommand{\Gal}{{\mathrm{Gal}}}
\newcommand{\Height}{{\mathrm{H}}}

\newcommand{\Heightp}{{\mathrm{H_p}}}

\newcommand{\genus}{{\mathbf g}}

\newcommand{\ualpha}{{\underline\alpha}}

\DeclareMathOperator{\Disc}{Disc}

\DeclareMathOperator{\Clg}{Cl}

\DeclareMathOperator{\rank}{rk}

\DeclareMathOperator{\Spec}{Spec}

\DeclareMathOperator{\innr}{innr}

\DeclareMathOperator{\Hom}{Hom}

\DeclareMathOperator{\Vol}{Vol}

\DeclareMathOperator{\tors}{tors}

\newcommand{\G}{\mathcal{G}}

1

\makeatletter

\renewcommand*\l@section[2]{%
  \ifnum \c@tocdepth >\z@
    \addpenalty\@secpenalty
    \addvspace{0.2em \@plus\p@}%
    \setlength\@tempdima{1.5em}%
    \begingroup
      \parindent \z@ \rightskip \@pnumwidth
      \parfillskip -\@pnumwidth
      \leavevmode \bfseries
      \advance\leftskip\@tempdima
      \hskip -\leftskip
      #1\nobreak\hfil \nobreak\hb@xt@\@pnumwidth{\hss #2}\par
    \endgroup
  \fi}
  
 \makeatother

\begin{document}

\hbadness 1500

\hfuzz 4pt

\title{Chevalley-Weil Theorem and Subgroups of Class Groups}

\author{Yuri Bilu \and Jean Gillibert}

\date{September 2017}

\maketitle

\begin{abstract}
We prove, under some mild hypothesis, that an {\'e}tale cover of curves defined over a number field has infinitely many specializations into an everywhere unramified extension of number fields. This constitutes an ``absolute'' version of the Chevalley-Weil theorem.
Using this result, we are able to generalise the techniques of Mestre, Levin and the second author for constructing and counting number fields with large class group.
\end{abstract}

%%%%%%%%%%%%%%%%%%%%%%%%%%%%%%%%%%%%%%%%%%%%%

%%%%%%%%%%%%%%%%%%%%%%%%%%%%%%%%%%%%%%%%%%%%%

\begin{footnotesize}
\tableofcontents
\end{footnotesize}

%%%%%%%%%%%%%%%%%%%%%%%%%%%%%%%%%%%%%%%%%%%%%

\section{Introduction}

\textsl{In this article, ``curve'' always stands for a ``smooth geometrically irreducible projective curve''.}

\bigskip

Let~$H$ be a finite abelian group,~$K$ a number field and ${d>1}$ an integer. The following conjecture is widely believed to be true.

\begin{conjecture}
\label{cong}
The field~$K$ has infinitely many extensions~$L$ of degree ${[L:K]=d}$ such that~$H$ is a subgroup of the class group $\Clg(L)$. 
\end{conjecture}

This conjecture is known to be true in many special cases, for instance,  when ${K=\Q}$ and~$H$ is a cyclic group or a product of two cyclic groups, see~\cite{AI84}. However, the general case remains widely  open.  We refer to \cite{BL05,GL12,Le07} for the history of the problem and further references. 

It is  clear that one may restrict to the case when ${H=(\Z/n)^r}$, the product of~$r$ cyclic groups of order~$n$, where~$n$ and~$r$ are positive integers. Denote by $\rank_n M$ the $n$-rank of a finite abelian group~$M$; that is,   the maximal integer~$r$ such that  ${(\Z/n)^r\le M}$. Then Conjecture~\ref{cong} is equivalent to the following.

\begin{conjecture}
\label{conj2}
Let ${n>1}$ be an integer. Then $\rank_n\Clg(L)$ is unbounded when~$L$ runs through the extensions of~$K$ of degree ${[L:K]=d}$. 
\end{conjecture}

When $n=d$, and more generally when $n$ divides $d$, this conjecture follows easily from Class Field Theory. On the other hand, when $n$ and $d$ are coprime, there is not a single case where Conjecture~\ref{conj2} is known to hold. For example, given $n>1$, there exists infinitely many imaginary quadratic fields such that $\rank_n\Clg(L)$ is at least $2$. For $n\geq 7$ this is currently the best known result on $n$-ranks of quadratic fields.

All partial results towards this conjecture  were obtained by the same strategy: one considers a suitable curve~$\CC$ over~$K$ admitting a finite $K$-morphism ${\CC\to \P^1}$ of degree~$d$. Hilbert's Irreducibility Theorem implies abundance of ${P\in \CC(\bar K)}$ such that ${[K(P):K]=d}$. When~$\CC$ has many independent étale Galois covers with cyclic group of order $n$, ``most'' of the fields $K(P)$ have a class group with large $n$-rank.

In particular, systematizing and generalizing the previous work, most notably the results of~\cite{Le07}, the following theorem was proved in~\cite{GL12} in the case when ${K=\Q}$ and $\CC$ is a superelliptic curve defined by a ``nice equation'' (see Corollary~3.1 of \cite{GL12}).

\begin{theorem}
\label{thold}
Let~$\CC$ be a  curve over a number field~$K$, let $J(\CC)$ be the Jacobian of~$\CC$, and let ${n>1}$ be an integer. Assume that~$\CC$ admits a finite $K$-morphism ${\CC \to \P^1}$ of degree~$d$, totally ramified over some point belonging to $\P^1(K)$.
Then there exist infinitely many (isomorphism classes of) number fields~$L$ with ${[L:K]=d}$  such that
\begin{equation}
\label{eold}
\rank_n \Clg(L) \geq \rank_n J(\CC)(K)_{\tors} - \rank_{\Z} \mathcal{O}_{L}^\times + \rank_{\Z} \mathcal{O}_K^\times + \rank_n \Clg(K).
\end{equation}
\end{theorem}

The  proof of this theorem  is based on Kummer theory and can be found in Subsection~\ref{ssproofs}. The last sentence can be made quantitative; see Theorem~\ref{thquant} below.

While  Theorem~\ref{thold} is quite general, its applicability in concrete cases is impaired by the presence of the negative term 
${- \rank_{\Z} \mathcal{O}_{L}^\times}$ on the right: the rank of the unit group of the field $L$ is quite large, especially if ${K\ne \Q}$.

The principal result of the present article is the following theorem, where this deficiency is avoided. Denote by $\rank_{\mu_n}J(\CC)$ the maximal integer~$r$ such that $J(\CC)$ has a $\Gal(\bar K/K)$-submodule isomorphic to $\mu_n^r$. In particular, if ${\mu_n\subset K}$  then  ${\rank_{\mu_n}J(\CC)=\rank_nJ(\CC)}$.

\begin{theorem}
\label{thnew}
In the set-up of Theorem~\ref{thold},
there exist infinitely many number fields~$L$ with ${[L:K]=d}$ such that
\begin{equation}
\label{enew}
\rank_n \Clg(L) \geq \rank_{\mu_n} J(\CC) + \rank_n \Clg(K).
\end{equation}
\end{theorem}

The proof is based on Class Field Theory (see Subsection~\ref{ssproofs}). It can be seen as a generalization of the construction of Mestre~\cite{Me92}. In Section~\ref{s_applications}, we revisit Mestre's result in the light of Theorem~\ref{thnew}.

In both Theorems~\ref{thold} and~\ref{thnew} the ``infinitely many'' can be made quantitative. 
Given a number field extension $L/K$, we define 
$$
\DD(L/K) =\bigl|\norm_{K/\Q}\Delta(L/K)\bigr|^{1/[K:\Q]},
$$
where $\Delta(L/K)$ is the  discriminant of~$L$ over~$K$. 

\begin{theorem}
\label{thquant}
Let ${t\in K(\CC)}$ be the rational function defining  the $K$-morphism ${\CC\to \PP^1}$ appearing in both Theorems~\ref{thold} and~\ref{thnew}. 
Assume that there exists a rational function ${x\in K(\CC)}$ of degree~$m$ such that ${K(\CC)=K(t,x)}$. Then, for sufficiently large positive~$X$,
in both these theorems  the number of (isomorphism classes of) the fields~$L$ satisfying~\eqref{eold} or~\eqref{enew}, respectively, and such that ${\DD(L/K)\le X}$ is at least $cX^{\ell/(2m(d-1))}/\log X$. Here ${\ell=[K:\Q]}$ and ${c>0}$ depends on~$\CC$,~$t$,~$x$ and~$K$. 
\end{theorem}

\begin{remark}
Assuming that~$\CC$ has a $K$-rational point (which is the case in both  Theorems~\ref{thold} and~\ref{thnew}), the Theorem of Riemann-Roch implies that there is a function ${x\in K(\CC)}$   of degree ${m\le 2\genus(\CC)+1}$  and such that ${K(\CC)=K(t,x)}$. In concrete examples a suitable~$x$ of much smaller degree~$m$ can often be found. 
\end{remark}

We believe that Theorem~\ref{thnew} has more  potential, but we have few examples up to now. Here is one application.

\begin{theorem}
\label{thp}
Let ${p\geq 3}$ be a prime, and let~$d$ be an integer such that ${2 \leq d\leq p-1}$. Let~$K$ be a number field containing the $p$-th roots of unity. Then there exist infinitely many (non-isomorphic) extensions $L/K$ with $[L:K]=d$ such that
$$
\rank_p \Clg(L) \geq 3 + \rank_p \Clg(K).
$$
More precisely, for sufficiently large positive~$X$ the number of such~$L$ with ${\DD(L/K)\le X}$ is at least ${cX^{(p-1)/2p(d-1)}/\log X}$, where ${c=c(K,p)>0}$. 
\end{theorem}

For instance, the cyclotomic field $\Q(\zeta_p)$ has infinitely many quadratic extensions $L/\Q(\zeta_p)$ such that 
${\rank_p \Clg(L) \geq 3}$. 

{\sloppy

Moreover, we know that there exist infinitely many irregular primes,  for which ${\rank_p \Clg(\Q(\zeta_p)) \geq 1}$. If we take ${p=37}$ for example, our result says that there exist infinitely many quadratic extensions $L/\Q(\zeta_{37})$ such that ${\rank_{37} \Clg(L) \geq 4}$.

 }

Our main tool in proving Theorem~\ref{thnew} is the Chevalley-Weil theorem. 
Recall that the (one-dimensional) Chevalley-Weil theorem asserts the following.

\begin{theorem}
\label{thchw}
Let ${\psi:\tilCC\to \CC}$ be an {\'e}tale morphism of curves over a number field~$K$. Then there exists a finite set~$S$ of places of~$K$ with the following property: for every ${P\in \CC(\bar{K})}$ and ${\tilP\in \tilCC(\bar{K})}$ such that ${\psi(\tilP)=P}$, the number field extension ${K(\tilP)/K(P)}$ is unramified at all finite places except perhaps those above~$S$. 
\end{theorem}

See \cite[Section~4.2]{Se97} for a quick proof and~\cite{BSS13} for several quantitative versions.

Following an idea of Mestre~\cite{Me92} (see ``Remarque'' on p.~372), we are going to show that, under certain assumptions, for ``many'' such~$P$ and~$\tilP$ the extension ${K(\tilP)/K(P)}$ is unramified everywhere, that is, at all places including the archimedean ones.

If~$X$ is a finite set of points on a $K$-variety, we denote by $K(X)$ the smallest field over which all points of~$X$ are rational. In particular, if we start from some point ${P\in \CC(K)}$, then $K(\psi^{-1}(P))$ is the compositum of fields $K(\tilP)$ where $\tilP$ runs through all points in $\tilCC(\bar{K})$ such that $\psi(\tilP)=P$.

We prove the following.

\begin{theorem}[Absolute Chevalley-Weil Theorem]
\label{thany}
In the set-up of Theorem~\ref{thchw}, assume that there exists a point ${A\in \CC(K)}$ such that 
 ${K(\psi^{-1}(A))=K}$. Then there exist infinitely many points ${P\in \CC(\bar{K})}$
 such that the extension ${K(\psi^{-1}(P))/K(P)}$ is unramified everywhere, that is, at all places including the archimedean ones.

More precisely, let ${t\in K(\CC)}$ be a rational function having~$A$ as its single zero\footnote{Existence of such~$t$ easily follows from the Riemann-Roch Theorem.}, and let $S$ be the set of places from Theorem~\ref{thchw}. Then there exists ${\eps>0}$ such that, for every ${P\in \CC(\bar{K})}$ satisfying ${t(P)\in K}$ and ${|t(P)|_v<\eps}$ for all ${v\in S\cup M_K^{\infty}}$, the extension ${K(\psi^{-1}(P))/K(P)}$ is unramified everywhere.
\end{theorem}

In the statement above, $M_K^{\infty}$ denotes the set of archimedean places of $K$. We refer to Section~\ref{s_DZ} for further notation.

\begin{remark}
\begin{enumerate}
\item In fact, we prove a stronger statement, namely that primes above $S$ are totally split in the extension ${K(\psi^{-1}(P))/K(P)}$.
\item The points $P$ whose existence is ensured by Theorem~\ref{thany} belong to the inverse image of $\PP^1(K)$ by $t$, in particular they satisfy $[K(P):K]\leq \deg(t)$.
\item It is certainly possible to prove a similar theorem in the case when the curves $\CC$ and $\tilCC$ are replaced by smooth projective varieties of arbitrary dimension.
\end{enumerate}
\end{remark}

\paragraph*{Plan of the article}
In Section~2 we prove the Absolute Chevalley-Weil Theorem.

In Section~\ref{s_DZ} we establish a counting result for number fields in fibers of a morphism ${\CC\to\PP^1}$, following, mainly, Dvornicich and Zannier. This will be our main tool in obtaining the quantitative results like Theorem~\ref{thquant}. 

In Section~\ref{s_torsors} we study specializations of torsors over algebraic curves, and prove Theorems~\ref{thold},~\ref{thnew} and~\ref{thquant}.

In the final Section~\ref{s_applications} we obtain some applications of our general results; in particular, we prove Theorem~\ref{thp}. 

\paragraph*{Acknowledgments}
A substantial part of this article was written when Yuri Bilu was visiting the Max-Planck-Institut für Mathematik in Bonn. He thanks this institute for the financial support and stimulating working conditions.

We thank Qing Liu for many stimulating discussions.  We thank Lev Birbrair and  Martin Widmer for  helpful suggestions. We also thank Ted Chinburg for sharing his valuable comments. Finally, we thank the anonymous referee for the encouraging report and detecting many inaccuracies, some quite deeply hidden.

%%%%%%%%%%%%%%%%%%%%%%%%%%%%%%%%%%%%%%%%%%%%%

%%%%%%%%%%%%%%%%%%%%%%%%%%%%%%%%%%%%%%%%%%%%%

\section{The Absolute Chevalley-Weil Theorem}

\label{sabschw}

%%%%%%%%%%%%%%%%%%%%%%%%%%%%%%%%%%%%%%%%%%%%%

In this section we prove Theorem~\ref{thany}. The idea of the proof is as follows: we take~$S$ to be the set whose existence is ensured by Theorem~\ref{thchw}. Then, using a variant of Hensel's Lemma, we prove that, if~$P$ is $v$-adically close enough to~$A$ with respect to some place ${v\in S}$, then ${K(\psi^{-1}(P))/K(P)}$ is unramified at~$v$.

In Subsection~\ref{ssloc} we collect the local results we need. Theorem~\ref{thany} is proved in Subsection~\ref{sstany}.

\subsection{Local Lemmas}

\label{ssloc}

In this subsection~$K$ is a complete field $(K,|\cdot|)$ (it might be archimedean). The absolute value   has a unique extension to the algebraic closure~$\bar{K}$, that we denote also $|\cdot|$ by abuse of notation.

%\textbf{Le cas archimédien est pour l'instant exclu (pour simplifier).}

%The following lemma was proved

\begin{lemma}[``Separable Hensel's lemma'']
\label{lkras}
Let ${f(X)\in K[X]}$ be a monic polynomial of degree~$n$ having exactly~$n$ distinct roots in~$\bar{K}$. Then there exists ${\eps>0}$ such that the following holds : for every monic polynomial $f^*(X)\in K[X]$ of degree~$n$ in the $\varepsilon$-neighbourhood of $f$, one has an isomorphism
$$
K[X]/f^*(X) \simeq K[X]/f(X).
$$
In other words, one may write $f(X)=\prod_{i=1}^n (X-\alpha_i)$ and $f^*(X)=\prod (X-\alpha_i^*)$ such that $K(\alpha_i)=K(\alpha_i^*)$ for all $i$.
\end{lemma}

(The $\eps$-neighbourhood is considered with respect to the $\ell_\infty$-norm.)

\begin{proof}
In the non-archimedean case the proof can be found in \cite[Theorem 1]{Brink06}. In the archimedean case this follows from the well-known fact (see, for instance, \cite[1.5.9]{BR90}) that on the space of monic separable real polynomials of degree~$n$ the function ``number of real roots'' is locally constant.
\end{proof}

\medskip

Now let~$\CC$ be a smooth projective $K$-curve and let ${t\in K(\CC)}$ be a rational function on $\CC$ having ${A\in \CC(K)}$ as its single zero. We denote by~$N$ the order of this zero, which is none other than the degree of ${t:\CC\to \P^1}$.

\begin{lemma}[``Puiseux expansion'']
\label{lpuiseux}
Let ${x\in \bar{K}(\CC)}$ be a $\bar{K}$-rational function on~$\CC$ not having a pole at~$A$. Then there exists ${\eps>0}$ such that for every ${P\in \CC(\bar{K})}$ with  ${|t(P)|<\eps}$ we have 
$$
x(P)=x(A)+O(|t(P)|^{1/N}),
$$
where the implicit constant may depend on~$\CC$,~$A$,~$t$ and~$x$, but not on~$P$. 
\end{lemma}

\begin{proof}
The lemma is an easy consequence of the following formally weaker statement.

\paragraph*{Claim}
Under the same hypothesis, there exists ${\eps>0}$ such that the function ${P\mapsto |x(P)|}$ is bounded on the set ${P\in \CC(\bar K)}$ satisfying ${|t(P)|<\eps}$. 

\medskip

The claim is trivial in the case ${x\in \bar K(t)}$. Indeed, if ${x=a(t)/b(t)}$ with ${a(T),b(T)\in \bar K[T]}$ coprime polynomials, then ${b(0)\ne 0}$ because~$x$ has no pole at~$A$. Hence there exits ${\eps>0}$ such that the set ${\{|b(\alpha)| : \alpha \in \bar K, |\alpha|<\eps\}}$ is separated away from~$0$, which immediately implies the claim. 

To prove the claim in general, recall the following well-known fact:  given a monic polynomial ${f(X)\in \bar K[X]}$  whose coefficients in absolute value are bounded by some ${C>0}$, the roots of~$f$ are bounded in absolute value by $2C$ (or even by~$C$ in the non-archimedean case); see, for instance, \cite[Proposition~3.2]{BB13}. In fact, we do not need such a precise statement: it suffices to know that when~$f$ runs over a set of monic polynomials with bounded coefficients, its roots stay bounded as well.

Now let 
${X^m+u_{m-1}X^{m-1}+\cdots+u_0\in \bar K(t)[X]}$ 
be the minimal polynomial of~$x$ over the field $\bar K(t)$. Since~$x$ has no pole at~$A$, none of the coefficients ${u_0, \ldots, u_{m-1}}$ does. Since they belong to $\bar K(t)$, the claim holds true for them. Since $x(P)$ is a root of the polynomial 
${X^m+u_{m-1}(P)X^{m-1}+\cdots+u_0(P)}$, 
the validity of the claim for~$x$ follows from that for ${u_0, \ldots, u_{m-1}}$ due to the property of polynomials quoted in the previous paragraph. This proves our claim. 

\medskip

The lemma follows immediately by applying the claim to the function ${(x-x(A))^N/t}$.
\end{proof}

\begin{lemma}[``Main Lemma'']
\label{lmain}
Let~$\CC$,~$A$,~$t$ be as above, and let ${\psi:\tilCC\to \CC}$ be a finite morphism of curves over~$K$ unramified above~$A$.
Then there exists  ${\eps>0}$ such that 
$$
\psi^{-1}(P) \simeq \psi^{-1}(A)\otimes_K K(P)
$$
for every ${P\in \CC(\bar{K})}$ satisfying ${|t(P)|<\eps}$.
\end{lemma}

\begin{proof}
Let $U_0=\Spec(R_0)$ be an affine open subset of $\CC$ containing $A$. We note that $R_0$ is a Dedekind ring with fraction field $K(\CC)$. Moreover, $V_0:=\psi^{-1}(U_0)=\Spec(\widetilde{R_0})$, where $\widetilde{R_0}$ is the normalization of $R_0$ in $K(\tilCC)$.

If $Q$ is a closed point of $V_0$, we denote by $\mathfrak{M}_Q$ the corresponding maximal ideal of $\widetilde{R_0}$. In particular, the residue field $\widetilde{R_0}/\mathfrak{M}_Q$ is the field $K(Q)$.

The map $\psi$ being unramified above $A$, we have:
$$
\sum_{Q\in \psi^{-1}(A)} [K(Q):K] = \deg(\psi)
$$
where the sum runs through closed points of $\psi^{-1}(A)$.

Let us choose in each $K(Q)$ an element $\beta_Q \in K(Q)$ such that the l.c.m. of their minimal polynomials over $K$ has degree $\deg(\psi)$. This can be acheived according to the equality above.

According to the Chinese remainder theorem, the map
$$
\widetilde{R_0} \to \prod_{Q\in \psi^{-1}(A)} K(Q)
$$
is surjective. Therefore, there exists a function $y\in \widetilde{R_0}$ such that $y\equiv\beta_Q\pmod{\mathfrak{M}_Q}$ for all $Q$. This means that $y$ takes the value $\beta_Q$ at $Q$.

We claim that $K(\tilCC)=K(\CC)(y)$. Let us observe first that $y$ belongs to $\widetilde{R_0}$ which is a finite $R_0$-algebra, hence $y$ is integral over $R_0$. We let ${f\in R_0[Y]}$ be the (monic) minimal polynomial of~$y$ over $R_0$.

For any closed point $P\in U_0$, we let $f_P$ be the image of $f$ by the reduction map $R_0[Y]\to K(P)[Y]$. In classical language, $f_P$ is the polynomial obtained by specializing the coefficients of $f$ at the point $P$.

By definition, $f(y)=0$ holds true in the ring $\widetilde{R_0}$, therefore it holds true after reduction modulo any ideal of $\widetilde{R_0}$. By construction of $y$, we know that $y\equiv\beta_Q\pmod{\mathfrak{M}_Q}$ for all $Q\in \psi^{-1}(A)$. Therefore, $f_A(\beta_Q)=0$ holds true for all $Q\in \psi^{-1}(A)$. In other words, the values of $y$ at points from $\psi^{-1}(A)$ are roots of $f_A$.

It follows that the l.c.m. of the minimal polynomials of the $\beta_Q$ divides $f_A$, hence $f_A$ (and therefore $f$) has degree at least equal to $\deg(\psi)$. But we already know that the degree of $f$ is at most $[K(\tilCC):K(\CC)]$, because $K(\CC)(y)$ is a subfield of $K(\tilCC)$. Hence the equality holds, and $y$ generates $K(\tilCC)$ over $K(\CC)$.

We can resume the situation by the following commutative diagram
$$
\begin{CD}
V_0=\psi^{-1}(U_0) @>\psi>> U_0 \\
@VVV @| \\
\Spec(R_0[Y]/f) @>>> \Spec(R_0) \\
\end{CD}
$$
in which the vertical map on the left is generically an isomorphism. In fact, the affine curve $V_0$ being normal, it is equal to the normalization of $\Spec(R_0[Y]/f)$.

Let $\Disc(f)$ be the discriminant of the polynomial $f$, and let $R:=R_0[\Disc(f)^{-1}]$. Thus, $U:=\Spec(R)$ is the largest open subset of $U_0$ over which $\Disc(f)$ is invertible. We note that $U$ contains $A$, because $f_A$ has no double root by construction. By standard algebra, $R[Y]/f$ is an {\'e}tale $R$-algebra. The ring $R$ being regular, it follows that $R[Y]/f$ is regular, hence normal. So, the map $\psi^{-1}(U)\to \Spec(R[Y]/f)$ is an isomorphism.

Therefore, given any $P\in U(\bar{K})$, we have
$$
\psi^{-1}(P)=\Spec(K(P)[Y]/f_P).
$$
According to Lemma~\ref{lpuiseux}, when $|t(P)|$ is sufficiently small, the coefficients of $f_P$ are sufficiently close to the coefficients of $f_A$. Hence, according to Lemma~\ref{lkras}, if $|t(P)|$ is sufficiently small, then
$$
K(P)[Y]/f_P \simeq K(P)[Y]/f_A \simeq (K[Y]/f_A)\otimes_K K(P)
$$
i.e.
$$
\psi^{-1}(P) \simeq \psi^{-1}(A)\otimes_K K(P)
$$
hence the result.
\end{proof}

\begin{remark}
\label{rmain}
The following more general version of the ``Main Lemma'' can be proved similarly. Let  ${\psi:\tilCC\to \CC}$ be a finite morphism of curves over~$K$.  Assume that there exists a non-constant function ${t\in K(\CC)}$  such that:
\begin{enumerate}
\item[i)] the zeroes of $t$ are $K$-rational points $A_1,\dots A_r$;
\item[ii)] the morphism~$\psi$ is unramified above the $A_i$;
\item[iii)] for any $i,j$, we have an isomorphism $\psi^{-1}(A_i)\simeq \psi^{-1}(A_j)$.
\end{enumerate}
Then a similar conclusion to that of Lemma~\ref{lmain} holds.
\end{remark}

%%%%%%%%%%%%%%%%%%%%%%%%%%%%%%%%%%%%%%%%%%%%%

\subsection{Proof of Theorem~\ref{thany}}
\label{sstany}

According to Theorem~\ref{thchw}, there exists a finite set~$S$ of places of~$K$  such that for every ${P\in \CC(\bar{K})}$ and ${\tilP\in \psi^{-1}(P)}$  the number field extension ${K(\tilP)/K(P)}$ is unramified outside places above~$S$. This is equivalent to saying that the extension  $K(\psi^{-1}(P))/K(P)$ is unramified outside places above $S$.

Now fix a place ${v\in S\cup M_K^{\infty}}$. Applying Lemma~\ref{lmain} over~$K_v$, we find that there exists ${\eps_v>0}$ such that, for every ${P\in \CC(\bar{K_v})}$ with $t(P)\in K$ and ${|t(P)|_v<\eps_v}$, we have, for any place $w$ of $K(P)$ lying above $v$,
$$
\psi^{-1}(P)\otimes_{K(P)} K(P)_w \simeq \psi^{-1}(A)\otimes_K K(P)_w.
$$
In particular, these finite varieties have the same function fields. But by assumption we have $K(\psi^{-1}(A))=K$, hence this yields $K(P)_w(\psi^{-1}(P)) =  K(P)_w$, in other words $\psi^{-1}(P)$ has all its points defined over $K(P)_w$. This means that $w$ is totally split in $K(\psi^{-1}(P))$. In particular, the extension $K(\psi^{-1}(P))/K(P)$ is unramified at $w$.

We complete the proof setting ${\eps=\min\{\eps_v: v\in S\cup M_K^{\infty}\}}$.

\begin{remark}
Replacing Lemma~\ref{lmain}  by Remark~\ref{rmain}, one can prove
a more general statement. In the set-up of  Theorem~\ref{thchw} assume that there exists  ${t\in K(\CC)}$ with the following property:  for any point~$A\in\CC(\bar{K})$ with ${t(A)=0}$ we have ${K(A)=K(\psi^{-1}(A))=K}$. Then the conclusion of Theorem~\ref{thany} holds.
\end{remark}

%%%%%%%%%%%%%%%%%%%%%%%%%%%%%%%%%%%%%%%%%%%%%

%%%%%%%%%%%%%%%%%%%%%%%%%%%%%%%%%%%%%%%%%%%%%

\section{Counting Number Fields in Fibers: the Theorem of Dvornicich \& Zannier}
\label{s_DZ}

%%%%%%%%%%%%%%%%%%%%%%%%%%%%%%%%%%%%%%%%%%%%%

Unless the contrary is stated explicitly, everywhere in this section
\begin{itemize}
\item
$K$ is a number field of degree~$\ell$ over~$\Q$,
\item
$\CC$ is a (smooth geometrically irreducible projective) curve over~$K$,
\item
${t:\CC\to\P^1}$ is a finite $K$-morphism of degree~$d$.
\end{itemize}
According to the Hilbert Irreducibility Theorem (see Subsection~\ref{sshilb}), for ``most'' ${\alpha\in K}$ the fiber ${t^{-1}(\alpha)}$ is $K$-irreducible. For our purposes we need a more precise statement: first, we have to consider not the entire field~$K$, but a proper subset, and second, we need to know that among the fields generated by the fibers there are ``many'' distinct.

We normalize the absolute values on number fields to extend the standard absolute values on~$\Q$: if ${v\mid \infty}$ then ${|2016|_v=2016}$, and if ${v\mid p<\infty}$ then ${|p|_v=p^{-1}}$. We denote by~$M_L$ the set of all absolute values on the number field~$L$ normalized as above, and by $M_L^\infty$ and $M_L^0$ the sets of infinite and of finite absolute values, respectively. 

We denote by $\Height(\alpha)$ the multiplicative absolute height of an algebraic number~$\alpha$: if~$L$ is a number field containing~$\alpha$ then
$$
\Height(\alpha) = \prod_{v\in M_L}\max\{1,|\alpha|_v \}^{[L_v:\Q_v]/[L:\Q]}. 
$$
We will use the standard properties of heights like 
\begin{equation}
\label{eprohe}
\Height(\alpha+\beta)\le 2\Height(\alpha)\Height(\beta), \quad 
\Height(\alpha\beta)\le \Height(\alpha)\Height(\beta), \quad
\Height(\alpha^n)= \Height(\alpha)^{|n|},
\end{equation}
etc.

\begin{theorem}
\label{thdvorzann}
Let~$S$ be a finite set of places of~$K$ and~$\eps$  a positive real number.
Further, let~$\mho$ be a thin subset of~$K$ (see Subsection~\ref{sshilb}).  Then there exist positive numbers ${c=c(K,\CC,t,S,\eps)}$ and ${B_0=B_0(K,\CC,t,S,\eps,\mho)}$ such that, for every ${B\ge B_0}$ the following holds. Consider  the points ${P\in \CC(\bar K)}$  satisfying
\begin{align}
\label{etink}
t(P)&\in K\smallsetminus \mho,\\
\label{eteps}
|t(P)|_v&<\eps \qquad (v\in S),\\
%\label{eheb}
\Height(t(P))&\le B. \nonumber
\end{align}
Then among the number fields $K(P)$, where~$P$ satisfies the conditions above, 
there are at least $cB^\ell/\log B$ distinct fields of degree~$d$ over~$K$. 
\end{theorem}

This theorem is, essentially, due to Dvornicich and Zannier~\cite{DZ94}. In particular, the case ${K=\Q}$ (and arbitrary~$S$,~$\eps$) can be easily deduced from \cite[Theorem~2(a)]{DZ94}. For the general case we need a suitable generalization of the result of Dvornicich and Zannier, which can be found in~\cite{Bi16}, see Subsection~\ref{ssdvorzann}.

\begin{remark}
The estimate ${cB^\ell/\log B}$ is sufficient for us, but it is, probably, far from best possible. For instance, a result of Corvaja and Zannier \cite[Corollary~1]{CZ03} implies, for sufficiently large~$B$,  a lower bound of the form $B^\ell(\log B)^k$ with arbitrary ${k>0}$  provided~$t$ has at least~$3$ zeros in $\CC(\bar K)$. Using methods of article~\cite{Sc79}\footnote{Attention: in~\cite{Sc79} the height is normalized with respect to~$K$, and not with respect to~$\Q$, as in the present article.}, one can show that for sufficiently large~$B$ there are at least $c' B^{2\ell}$ numbers ${\alpha \in K}$ satisfying ${|\alpha|_v<\eps}$ for every ${v\in S}$, and ${\Height(\alpha)\le B}$; here ${c'>0}$ depends on~$K$,~$S$ and~$\eps$. Moreover, one can probably  even prove the asymptotics ${\gamma B^{2\ell}}$ (as ${B\to\infty}$) for the counting function of such numbers; here ${\gamma>0}$ depends  on~$K$,~$S$ and~$\eps$.  This suggests that a lower bound of the form $cB^{2\ell}/(\log B)^A$ must hold true with some ${A>0}$. 
\end{remark}

In Subsection~\ref{ssdiscrim} we estimate the discriminants of the fields emerging in Theorem~\ref{thdvorzann}.  Recall that, given a number field extension $L/K$, we define 
$$
\DD(L/K) =\bigl|\norm_{K/\Q}\Delta(L/K)\bigr|^{1/[K:\Q]},
$$
where $\Delta(L/K)$ is the  discriminant of~$L$ over~$K$.  %Recall that ${\ell=[K:\Q]}$. 
In Subsection~\ref{ssdiscrim} we estimate $\DD(K(P)/K)$, where~$P$ is as in Theorem~\ref{thdvorzann}; see Proposition~\ref{pdiscrim}.

Combining Theorem~\ref{thdvorzann} and Proposition~\ref{pdiscrim}, we obtain the following consequence.  

\begin{corollary}
In the set-up of Theorem~\ref{thdvorzann}, assume in addition that there exists a non-constant rational function ${x\in K(\CC)}$ of degree~$m$ such that
$$
{K(\CC)=K(t,x)}.
$$
Then there exist positive numbers
$$
{c=c(K,\CC,t,S,\eps)} \quad\text{and}\quad {X_0=X_0(K,\CC,t,S,\eps)}
$$
such that, for every ${X\ge X_0}$ there exist at least $cX^{\ell/(2m(d-1))}/\log X$ distinct fields~$L$ with  
$$
[L:K]=d, \qquad \DD(L/K)\le X, 
$$
and of the form ${L=K(P)}$, where~$P$ satisfies~\eqref{etink},~\eqref{eteps}. 
\end{corollary}

\subsection{Thin Subsets and Hilbert's Irreducibility Theorem}
\label{sshilb}

In this subsection we recall basic definitions and facts about thin sets, and state Hilbert's Irreducibility Theorem.

Let~$K$ be a field of characteristic~$0$.  We call  ${\mho \subset K}$  a \textsl{basic thin subset} of~$K$ if there exists a (smooth geometrically irreducible) curve~$\CC$ defined over~$K$ and a non-constant rational function ${u\in K(\CC)}$ of degree at least~$2$  such that ${\mho \subset u(\CC(K))}$. A \textsl{thin subset} of~$K$ is a union of finitely many basic thin subsets. Thin subsets form an ideal in the algebra of subsets of~$K$. Serre in \cite[Section~9.1]{Se97} gives a differently looking, but equivalent definition of thin sets.

Any finite set is thin, and if~$K$ is algebraically closed then any subset of~$K$ is thin. 

\begin{remark}
\label{rthin}
If~$L$ is an extension of~$K$ then any thin subset of~$K$ is also thin as a subset of~$L$. The converse is true when~$L$ is finitely generated over~$K$ \cite[Proposition~2.1]{BL05} but not in general; for instance, any number field~$K$ is a thin subset of its algebraic closure~$\bar K$ but is not a thin subset of itself by the Hilbert Irreducibility Theorem quoted below. 
\end{remark}

Using elementary Galois theory one easily proves the following (see \cite[Section~9.2]{Se97})

\begin{proposition}
\label{pgalo}
Let~$\CC$ be a curve over~$K$ and ${t\in K(\CC)}$ a non-constant rational function. Then the set of ${\alpha\in K}$ such that the fiber ${t^{-1}(\alpha)}$ is reducible over~$K$ is thin. 
\end{proposition}

Hilbert's Irreducibility Theorem asserts that when~$K$ is a number field then its ring of integers $\OO_K$ is not a thin subset of~$K$. In fact, one has the following counting result (see~\cite{Se97}, Theorem on page~134). 

\begin{theorem}
\label{tthin}
Let~$K$ be a number field of degree~$\ell$ over~$\Q$ and~$\mho$ a thin subset of~$K$. Then for ${B\ge 1}$ the set ${\mho\cap \OO_K}$ has at most $O(B^{\ell/2})$ elements~$\alpha$ satisfying ${|\alpha|_v\le B}$ for every ${v\in M_K^\infty}$; the implicit constant depends  on~$K$ and on~$\mho$. 
\end{theorem}

%%%%%%%%%%%%%%%%%%%%%%%%%%%%%%%%%%%%%%%%%%%%%

\subsection{Counting Algebraic Integers}

To prove Theorem~\ref{thdvorzann} we need to count algebraic integers in the number field~$K$ whose conjugates are bounded by given quantities. We denote by~$s_1$ and~$s_2$ the number of real and of complex infinite places of~$K$, so that ${\ell=s_1+2s_2}$ and ${|M_K^\infty|=s_1+s_2}$. We also denote by~$D_K$ the  discriminant of~$K$ over~$\Q$. 

\begin{proposition}
\label{pcountv}
For every ${v\in M_K^\infty}$ pick ${B_v\ge 1}$.  Then the total number of ${\alpha\in \OO_K}$ satisfying 
${|\alpha|_v\le B_v}$ for ${v\in M_K^\infty}$ is 
$$
\frac{2^{s_1+s_2}\pi^{s_2}\prod_vB_v^{e_v}}{|D_K|^{1/2}}\left(1+O\left(\frac1{\min\{B_v:v\in M_K^\infty\}}
\right)\right),
$$
where ${e_v=1}$ if~$v$ is real, ${e_v=2}$  if~$v$ is complex, and the implicit constant depends only on~$K$.
\end{proposition}

This is, of course, well-known and classical, but we did not find exactly this statement in the literature. Therefore we add some details for the reader's convenience.

Let us recall some standard facts from the Geometry of Numbers. If~$\Gamma$ is a lattice in~$\R^\ell$ and ${\UU\subset\R^\ell}$ a bounded symmetric convex set, then ${|\UU\cap\Gamma|}$ must be approximated by $\frac{\Vol\UU}{\det\Gamma}$, where $\Vol$ denotes the standard Euclidean volume on~$\R^\ell$ and $\det\Gamma$ is the fundamental volume of~$\Gamma$. To make this precise, we define the \textsl{inner radius} of~$\UU$ as the minimal (Euclidean) distance  from  the boundary of~$\UU$ to the origin: 
$$
\innr\UU=\min\{\|x\|_2: x\in \partial U \}.
$$

\begin{proposition}
\label{pgeo}
Let~$\Gamma$ be a lattice in~$\R^d$. Then %there exists a positive real number~$\kappa$, depending only on~$\Gamma$, such that 
for any bounded symmetric convex set~$\UU$ we have 
$$
|\UU\cap\Gamma|=\Vol\UU\left(\frac1{\det\Gamma}+O\left(\frac1{\innr\UU}\right)\right),
$$
where the implied constant may depend on~$\Gamma$, but not on~$\UU$. 
\end{proposition}

Proposition~\ref{pcountv} readily follows from this, by viewing~$\OO_K$ as a lattice in $\R^{s_1}\times\C^{s_2}$, of fundamental volume ${2^{-s_2}|D_K|^{1/2}}$. %The details are well-known, we leave them out. 

\bigskip

 %We view $\R^d$ as the Euclidean space with the standard inner product. 
Proving Proposition~\ref{pgeo} requires some preparation. The standard inner product on the Euclidean space~$\R^\ell$ induces an inner product on every subspace~$\LL$, and we denote by $\Vol_\LL$ the volume on~$\LL$  induced by this inner product. 

\begin{lemma}
\label{lproj}
Let~$\LL$ be a subspace of~$\R^\ell$ and denote by ${\pi_\LL:\R^\ell\to \LL^\perp}$ the orthogonal projection along~$\LL$. Then for any bounded symmetric convex set ${\UU\subset \R^d}$ we have
$$
\Vol_\LL(\UU\cap\LL)\Vol_{\LL^\perp}(\pi_\LL(\UU))\le \binom \ell m \Vol\UU,
$$
where ${m=\dim\LL}$. 
\end{lemma}

\begin{proof}
See \cite[Lemma~6.6]{Bi98}.
\end{proof}

\bigskip

%Let~$\UU$ be a bounded symmetric convex subset of~$\R^n$. 

The intersection ${\UU\cap\LL}$ contains the (open) $\ell$-dimensional ball of radius~$\innr\UU$. Hence Lemma~\ref{lproj} has the following consequence.

\begin{corollary}
\label{cproj}
In the set-up of Lemma~\ref{lproj} we have 
$$
\Vol_{\LL^\perp}(\pi_\LL(\UU))\le\frac c{(\innr\UU)^m}\Vol\UU,
$$
where~$c$ depends only on the dimension~$d$. 
\end{corollary}

%Now we are ready to prove Proposition~\ref{pgeo}. 
\begin{proof}[Proof of Proposition~\ref{pgeo}]
We may assume that ${\Gamma=\Z^\ell}$. Indeed, pick some basis of~$\Gamma$ and consider the new inner product, making this basis orthonormal.  The quantities ${|\UU\cap\Gamma|}$ and $\frac{\Vol\UU}{\det\Gamma}$ will be not altered, and the inner radius~$R$ will be replaced by ${R'\ge cR}$, where~$c$ depends only on the basis we picked. 

For a subset ${S\subseteq\{1,2,\ldots,\ell\}}$ let~$\LL_S$ be the subspace of $\R^\ell$ defined by ${x_i=0}$ for ${i\in S}$. 
By the classical result of Davenport~\cite{Da51}, 
$$
\bigl||\UU\cap\Z^\ell|-\Vol\UU\bigr| \le \sum_{S\ne\varnothing}\Vol_{\LL_S^\perp}(\pi_{\LL_S}(\UU)),
$$
the sum being over the non-empty subsets ${S\subseteq\{1,2,\ldots,\ell\}}$. By Corollary~\ref{cproj}, each summand on the right is bounded by  $c(\ell)\frac{\Vol\UU}{\innr\UU}$. This proves Proposition~\ref{pgeo}. 
\end{proof}

Here is an immediate consequence. 

\begin{corollary}
\label{ccount}
For  positive real numbers~$B$ and~$E$  satisfying ${B\ge E\ge 2}$,  
at most ${O(EB^{\ell-1})}$ numbers ${\alpha \in \OO_K}$ satisfy
$$
\max_{v\in M_K^\infty}|\alpha|_v\le B,   \qquad \min_{v\in M_K^\infty}|\alpha|_v\le E.
$$
The implicit constant depends only on~$K$. 
\end{corollary}

\subsection{Proof of Theorem~\ref{thdvorzann}}
\label{ssdvorzann}

The following result is Corollary~8.2 from~\cite{Bi16}.

\begin{theorem}
\label{thdvorzannb}
%Let~$K$,~$\CC$  and~$t$ be as in Theorem~\ref{thdvorzann}.    Then there 
There exist positive numbers ${c=c(K,\CC,t)}$ and ${B_0=B_0(K,\CC,t)}$ such that, for every ${B\ge B_0}$, the following holds. Consider  the points ${P\in \CC(\bar K)}$  satisfying
\begin{align}
\label{einok}
t(P) &\in \OO_K,\\
\label{etpb}
|t(P)|_v&\le B \qquad (v \in M_K^\infty).
\end{align} 
Then among the number fields $K(P)$, where~$P$ satisfies the conditions above, 
there are at least $cB^\ell/\log B$ distinct fields.% of degree~$n$ over~$K$. 
\end{theorem}

Combining this with Theorem~\ref{tthin} and  Corollary~\ref{ccount}, we obtain the following statement.

\begin{corollary}
\label{cdvorzannbb}
%Let~$K$,~$\CC$  and~$t$ be as in Theorem~\ref{thdvorzann}, and let 
Let ${E\ge 2}$ be a real number and let~$\mho$ be a thin subset of~$K$.     Then there exist positive numbers ${c=c(K,\CC,t)}$ and ${B_0=B_0(K,\CC,t,E,\mho)}$ such that, for every ${B\ge B_0}$, the following holds. Consider  the points ${P\in \CC(\bar K)}$  satisfying
\begin{align}
\label{einokbb}
t(P) &\in \OO_K\smallsetminus\mho,\\
\label{etpe}
|t(P)|_v&\ge E \qquad (v \in M_K^\infty),\\
\label{ehebbb}
\Height(t(P))&\le B.
\end{align} 
Then among the number fields $K(P)$, where~$P$ satisfies the conditions above, 
there are at least $cB^\ell/\log B$ distinct fields of degree~$d$ over~$K$. 
\end{corollary}

\begin{proof}
Corollary~\ref{ccount}  implies that there exists at most $O(EB^{\ell-1})$ points~$P$ for which~\eqref{einok} and~\eqref{etpb} hold, but~\eqref{etpe} does not hold.  Denote by~$c'$ and $B_0'$  the numbers~$c$ and $B_0$ from Theorem~\ref{thdvorzannb}.   Then for ${B\ge B_0'}$ we find   ${c'B^\ell/\log B-O(EB^{\ell-1})}$ points~$P$ satisfying~\eqref{etpb},~\eqref{einokbb} and~\eqref{etpe},  for which the fields 
$K(P)$ are pairwise distinct. Since 
${\Height(\alpha)\le \max_{v\in M_K^\infty} |\alpha|_v}$
for ${\alpha\in \OO_K}$, all these points satisfy~\eqref{ehebbb}.

By Proposition~\ref{pgalo} and Theorem~\ref{tthin}, only $O(B^{\ell/2})$ of these points~$P$ satisfy ${t(P)\in \mho}$ or ${[K(P):K]<d}$.  
It remains to observe that for sufficiently large~$B$ we have 
$$
c'B^\ell/\log B-O(EB^{\ell-1})-O(B^{\ell/2}) \ge (c'/2) B^\ell/\log B. 
$$
This proves Corollary~\ref{cdvorzannbb} with ${c=c'/2}$.% and ${B_0=\max\{B_0',c''E\log E\}}$.
\end{proof}

\bigskip

Now we are ready to complete the proof of Theorem~\ref{thdvorzann}. 
Applying a suitable coordinate change, we  reduce it to Corollary~\ref{cdvorzannbb}.

Let~$S$ and~$\eps$ be as in Theorem~\ref{thdvorzann}. Pick a non-zero ${a\in \OO_K}$ and a  real ${E\ge 2}$ satisfying 
\begin{align*}
|a|_v&<\min\{\eps,1\} && (v\in S^0),\\
E&> \eps^{-1}+|a|_v^{-1} && (v\in S^\infty),
\end{align*}
where~$S^0$ and~$S^\infty$ denote the sets of finite and of infinite places from~$S$. Set ${t^\ast = 1/t+1/a}$, so that ${t=1/(t^\ast-a^{-1})}$, and ${\mho^\ast=\{1/\tau+1/a: \tau\in \mho\}}$.   

Applying Corollary~\ref{cdvorzannbb} to the data~$K$,~$\CC$,~$t^\ast$,~$E$ and~$\mho^\ast$, for every ${B\ge B_0}$ we find $cB^\ell/\log B$ points~$P$ satisfying
\begin{align}
\label{einokbbstar}
t^\ast(P) &\in \OO_K\smallsetminus\mho^\ast,\\
\label{etpestar}
|t^\ast(P)|_v&\ge E \qquad (v \in M_K^\infty),\\
\Height(t^\ast(P))&\le B.
\nonumber
\end{align} 
and such that the fields $K(P)$ are pairwise distinct and of degree~$d$ over~$K$.

Due to our choice of~$a$ and~$E$, inequality~\eqref{eteps} follows from~\eqref{einokbbstar} for ${v\in S^0}$ and from~\eqref{etpestar} for ${v\in S^\infty}$. Also, ${t(P)\in\mho}$ if and only if ${t^\ast(P)\in \mho^\ast}$. Finally, using~\eqref{eprohe} we obtain 
$$
\Height(t(P))\le 2\Height(a)\Height(t^\ast(P))\le 2\Height(a) B.
$$
This proves Theorem~\ref{thdvorzann} with suitably adjusted~$c$ and~$B_0$. \qed

%%%%%%%%%%%%%%%%%%%%%%%%%%%%%%%%%%%%%%%%%%%%%

\subsection{Estimating the discriminants}
\label{ssdiscrim}

In this subsection we estimate the discriminants of number fields generated by irreducible fibers.  Recall that, given a number field extension $L/K$, we define 
$$
\DD(L/K) =\bigl|\norm_{K/\Q}\Delta(L/K)\bigr|^{1/\ell},
$$
where $\Delta(L/K)$ is the  discriminant of~$L$ over~$K$ and ${\ell=[K:\Q]}$. 

For a point ${P\in \CC(\bar K)}$ such that ${t(P)\in K}$, we will write $\DD(P)$ for $\DD(K(P)/K)$.

\begin{proposition}
\label{pdiscrim}
Assume that there exists a non-constant rational function ${x\in K(\CC)}$ of degree~$m$ such that ${K(\CC)=K(t,x)}$. 
Then for every point ${P\in \CC(\bar K)}$ such that 
\begin{equation}
\label{ehyp}
t(P)\in K, \qquad [K(P):K]=d
\end{equation}  
we have 
\begin{equation}
\label{ediscrim}
\DD(P) \le c \Height(t(P))^{2m(d-1)},
\end{equation}
where~$c$ depends on~$\CC$ and~$t$ (but not on~$K$). 
\end{proposition}

{\sloppy

For the proof we need some lemmas. Recall the notion of  projective height: if ${\ualpha=(\alpha_0, \ldots, \alpha_N)\in \P^N(\bar\Q)}$ then 
$$
\Heightp(\ualpha):=\prod_{v\in M_L}\max\{|\alpha_0|_v,\ldots,|\alpha_N|_v \}^{[L_v:\Q_v]/[L:\Q]},
$$
where~$L$ is a number field containing ${\alpha_0, \ldots, \alpha_N}$. %If ${\ualpha=(\alpha_1, \ldots, \alpha_N)\in \bar\Q^N}$ then$$\Heighta(\ualpha)=\prod_{v\in M_L}\max\{1,|\alpha_1|_v,\ldots,|\alpha_N|_v \}^{[L_v:\Q_v]/[L:\Q]}.$$
Note that ${\Height(\alpha)=\Heightp(1,\alpha)}$ for ${\alpha\in \bar\Q}$. 

}

The projective height of a polynomial with algebraic coefficients is the projective height of the vector of its non-zero coefficients. 

\begin{lemma}
\label{ldiscrim}
Let ${f(X)\in K[X]}$ be a $K$-irreducible 
polynomial of degree~$d$. Then for any of its roots ${\beta \in \bar K}$ we have 
$$
\DD(K(\beta)/K) \le d^{3d}\Heightp(f)^{2(d-1)}.
$$
\end{lemma}

\begin{proof}
Let ${\beta_1, \ldots, \beta_d}$ be the roots of~$f$. According to Corollary 3.17 from~\cite{BSS13}, 
$$
\prod_{i=1}^d\DD(K(\beta_i)/K)^{1/[K(\beta_i):K]}\le d^{3d}\Heightp(f)^{2(d-1)}. 
$$
However, since~$f$ is $K$-irreducible, all the ${\DD(K(\beta_i)/K)}$ are equal, and all the ${[K(\beta_i):K]}$  are equal to~$d$. Whence the result.
\end{proof}

\begin{lemma}
\label{lhpol}
Let ${F(T,X)\in \bar\Q[T,X]}$ be a polynomial of $T$-degree~$m$, and let ${\alpha\in \bar\Q}$. Then the polynomial ${f(X)=F(\alpha,X)}$ satisfies 
$$
\Heightp(f)\le (m+1)\Heightp(F)\Height(\alpha)^m.
$$
\end{lemma}

\begin{proof}
Let~$L$ be a number field containing~$\alpha$ and the coefficients of~$F$. For ${v\in M_L}$ denote by $|F|_v$ and $|f|_v$ the maximal  $v$-value of the coefficients of~$F$, respectively~$f$. We estimate trivially
$$
|f|_v \le 
\begin{cases}
|F|_v\max\{1,|\alpha|_v\}^m, & v\in M_L^0,\\
(m+1)|F|_v\max\{1,|\alpha|_v\}^m, & v\in M_L^\infty.
\end{cases}
$$
Multiplying over ${v\in M_L}$, the result follows.
\end{proof}

\begin{proof}[Proof of Proposition~\ref{pdiscrim}]
For all but finitely many points ${P\in \CC(\bar K)}$ we have 
$$
K(P)=K(t(P),x(P)).
$$ 
In particular, ${K(P)=K(x(P))}$ for all but finitely many points~$P$ satisfying ${t(P)\in K}$.

There exists a polynomial ${F(T,X)\in K[T,X]}$ such that 
$$
F(t,x)=0, \qquad \deg_TF=m, \quad \deg_XF=d.
$$
Now assume that~$P$ satisfies~\eqref{ehyp}, that ${K(P)=K(x(P))}$, and that~$P$ is not a pole of~$t$ and not a pole of~$x$.  Then the polynomial ${f_P(X)=F(t(P),X)}$ is $K$-irreducible, of degree~$d$, and has $x(P)$ as one of its roots. 
Using Lemmas~\ref{ldiscrim} and~\ref{lhpol}, we obtain
$$
\DD(P) \le d^{3d}\Heightp(f_P)^{2(d-1)}\le d^{3d}\bigl((m+1)\Heightp(F)\bigr)^{2(d-1)}\Height(t(P))^{2m(d-1)}.
$$
This proves~\eqref{ediscrim} for all but finitely many~$P$ satisfying~\eqref{ehyp}. By adjusting~$c$ we obtain it for all such~$P$.
\end{proof}

%%%%%%%%%%%%%%%%%%%%%%%%%%%%%%%%%%%%%%%%%%%%%

%%%%%%%%%%%%%%%%%%%%%%%%%%%%%%%%%%%%%%%%%%%%%

\section{Specialization of torsors}
\label{s_torsors}

%%%%%%%%%%%%%%%%%%%%%%%%%%%%%%%%%%%%%%%%%%%%%

In this section we consider a finite flat (not necessarily commutative) $\mathcal{O}_K$-group scheme $\G$. We denote by $H^1_{\et}(\mathcal{O}_K,\G)$ (resp. $H^1_{\fl}(\mathcal{O}_K,\G)$) the cohomology set which classifies {\'e}tale (resp. flat) $\G$-torsors over $\mathcal{O}_K$. We denote by $\G_K$ the generic fibre of $\G$, and by $H^1(K,\G_K)$ the (possibly non-abelian) Galois cohomology set $H^1(\Gal(\bar{K}/K),\G_K(\bar{K}))$.

Let us recall that the ``restriction to the generic fiber'' map $H^1_{\fl}(\mathcal{O}_K,\G)\to H^1(K,\G_K)$ is injective. Indeed, if $\xi$ is a flat $\G$-torsor over $\mathcal{O}_K$, then, by descent theory, $\xi$ is representable by a finite flat $\mathcal{O}_K$-scheme. Therefore, if the generic fiber of $\xi$ has a section, then the valuative criterion of properness implies that $\xi$ itself has a section. By consequence, we have a chain of inclusions
$$
H^1_{\et}(\mathcal{O}_K,\G) \subset H^1_{\fl}(\mathcal{O}_K,\G) \subset H^1(K,\G_K).
$$

%%%%%%%%%%%%%%%%%%%%%%%%%%%%%%%%%%%%%%%%%%%%%

\subsection{Local splitting of torsors}

We now define a set of cohomology classes which are locally trivial at all places in $S$. %that are $S$-split, that is,

\begin{definition}
\label{S_split_def}
If $S$ is a finite set of places of $K$, we let
$$
H^1_{\text{$S$-split}}(\mathcal{O}_K,\G) := \ker\left(H^1_{\fl}(\mathcal{O}_{K,S},\G) \to \prod_{v\in S} H^1_{\fl}(K_v,\G_{K_v})\right).
$$
\end{definition}

\begin{lemma}
\label{S_split_lemma}
The set $H^1_{\text{$S$-split}}(\mathcal{O}_K,\G)$ is a subset of $H^1_{\fl}(\mathcal{O}_K,\G)$. Moreover, if $\G$ is {\'e}tale over $\mathcal{O}_{K,S}$, then it is a subset of $H^1_{\et}(\mathcal{O}_K,\G)$.
\end{lemma}

\begin{proof}
It follows from \cite[Corollary~4.2]{Ces14} that, if $S$ is a finite set of places of $K$, the square
$$
\begin{CD}
H^1_{\fl}(\mathcal{O}_K,\G) @>>> H^1_{\fl}(\mathcal{O}_{K,S},\G) \\
@VVV @VVV \\
\prod_{v\in S} H^1_{\fl}(\mathcal{O}_{K_v},\G) @>>> \prod_{v\in S} H^1(K_v,\G_{K_v}) 
\end{CD}
$$
is cartesian, with injective horizontal maps. Hence the result.
\end{proof}

\begin{theorem}
\label{int_torsors}
Let $K$ be a number field, and let $\G$ be a finite flat (non necessarily commutative) $\mathcal{O}_K$-group scheme. Consider the following setting:
\begin{itemize}
\item
$\CC$ is a (smooth geometrically irreducible projective) curve over~$K$;
\item
$\psi:\tilCC\to \CC$ is a $\G_K$-torsor (where $\tilCC$ is geometrically irreducible);
\item
$A$ is a $K$-rational point of $\CC$ such that $\psi^{-1}(A)$ is the trivial $\G_K$-torsor;
\item
${t:\CC\to\P^1}$ is a finite $K$-morphism, having $A$ as its single zero.
\end{itemize}
Let $S$ be any finite set of places of $K$ which contains the set from Theorem~\ref{thchw}. Let also~$F$ be a finite extension of~$K$.  Then:
\begin{enumerate}
\item[1)] there exists $\eps>0$ such that, for every point ${P\in \CC(\bar{K})}$ satisfying ${t(P)\in K}$ and ${|t(P)|_v<\eps}$ for all ${v\in S}$, the torsor $\psi^{-1}(P)$ belongs to the subset
$$
H^1_{\text{$S$-split}}(\mathcal{O}_{K(P)},\G) \subset H^1(K(P),\G_K),
$$
in particular, $\psi^{-1}(P)$ extends into a flat $\G$-torsor over $\mathcal{O}_{K(P)}$ (or even an {\'e}tale torsor, up to enlarging the set $S$).
\item[2)] there exist infinitely many $P$ as in 1) such that $[K(P):K]=\deg(t)$ and $\psi^{-1}(P)$ is the spectrum of a field which is linearly disjoint from~$F$.
\end{enumerate}
More precisely, there exist positive numbers $c$ and $B_0$ such that, for every $B\geq B_0$, the following holds: among the number fields $K(P)$, where $P$ is as in~2) and $H(t(P))\leq B$, there are at least $cB^\ell/\log B$ distinct fields.
\end{theorem}

\begin{remark}
In general, given a finite $K$-group scheme $G$, it may not be possible to extend $G$ into a finite flat group scheme over $\mathcal{O}_K$, and, in case such a group scheme exists, it may not be unique. Nevertheless, if the set $S$ contains all places dividing the order of $G$, there exists at most one way to extend $G$ into a finite flat (automatically {\'e}tale) group scheme over $\mathcal{O}_{K,S}$. In the statement above, what really matters is the fact that the group scheme $\G_K$ can be extended into a finite flat group scheme $\G$ over $\mathcal{O}_K$, not the choice of $\G$.
\end{remark}

\begin{proof}
Let $\psi:\tilCC\to\CC$ be a $\G_K$-torsor. Then there exists a finite set $S$ of finite places of $K$ such that $\psi$ extends into a $\G$-torsor $\tilde{C}\to C$ between $\mathcal{O}_{K,S}$-schemes, where $C$ is a smooth projective model of $\CC$ over $\mathcal{O}_{K,S}$. We note that such an $S$ contains places of bad reduction of the curve $\CC$, and also additional places that we may consider as being places where $\psi$ has bad reduction. Up to enlarging $S$, we may assume that $\G$ is {\'e}tale over $\mathcal{O}_{K,S}$.

By projectivity of $C$ over $\mathcal{O}_{K,S}$, any point $P\in\CC(\bar{K})$ can be extended into a section $\Spec(\mathcal{O}_{K(P),S'})\to C$, where $S'$ is the set of places of $K(P)$ that lie above places in $S$. It follows that, for any $P\in\CC(\bar{K})$, $\psi^{-1}(P)$ belongs to the subset
$$
H^1_{\et}(\mathcal{O}_{K(P),S'},\G) \subset H^1(K(P),\G_K).
$$

Let us fix a place $v$ of $K$. Applying Lemma~\ref{lmain} over the completion~$K_v$, we find that there exists ${\eps_v>0}$ such that, for every ${P\in \CC(\bar{K_v})}$ satisfying $|t(P)|_v<\eps_v$, we have
$$
\psi^{-1}(P)\otimes_{K(P)} K_v(P) \simeq \psi^{-1}(A)\otimes_K K_v(P)
$$
where $K_v(P)$ is the smallest extension of $K_v$ over which $P$ is defined. This means that $\psi^{-1}(P)$ is the trivial torsor over $K_v(P)$, because by hypothesis $\psi^{-1}(A)$ is the trivial torsor.

Now, if we consider a point $P\in \CC(\bar{K})$ and a place $w$ of $K(P)$ lying above $v$, then $K(P)_w$ contains $K_v(P)$ as a subfield, hence if $|t(P)|<\eps_v$ then $\psi^{-1}(P)$ becomes trivial over $K(P)_w$.

If we set ${\eps=\min\{\eps_v: v\in S\}}$, then $\psi^{-1}(P)$ belongs to $H^1_{\text{$S$-split}}(\mathcal{O}_{K(P)},\G)$ for every ${P\in \CC(\bar{K})}$ satisfying ${t(P)\in K}$ and ${|t(P)|_v<\eps}$ for all ${v\in S}$. This completes the proof of the first (qualitative) statement. Let us finally note that, according to Lemma~\ref{S_split_lemma}, for such $P$ the torsor $\psi^{-1}(P)$ extends into an {\'e}tale $\G$-torsor over $\mathcal{O}_{K(P)}$.

To prove the second (quantitative) statement, we apply Theorem~\ref{thdvorzann}. Set ${\tilt=t\circ \psi}$ and define the set ${\mho\subset K}$ as follows: ${\mho=\mho_1\cup\mho_2}$, where 
\begin{align*}
\mho_1&=\{\tilt(Q): Q\in \tilCC(\bar K), \ \tilt(Q) \in K, \ [K(Q):K]<\deg\tilt\},\\
\mho_2&=\{\tilt(Q): Q\in \tilCC(\bar K), \ \tilt(Q) \in K, \ \text{$K(Q)$ is not linearly disjoint from~$F$}\}.
\end{align*}
Proposition~\ref{pgalo} implies that~$\mho_1$ is a thin set in~$K$. The set~$\mho_2$ is thin in~$K$ as well. Indeed, since~$\tilCC$ is geometrically irreducible, the degree of~$\tilt$ does not change if we extend the base field from~$K$ to~$F$. But, if $K(Q)$ is not linearly disjoint from~$F$, then 
$$
[F(Q):F]<[K(Q):K]\le \deg \tilt,
$$
which implies that~$\mho_2$ is thin in~$F$. Remark~\ref{rthin} implies that~$\mho_2$ is thin in~$K$ as well.
Thus, the set~$\mho$ is thin.

Theorem~\ref{thdvorzann} implies that there exists a  positive number~$c$ such that, for sufficiently large~$B$, among the number fields $K(P)$ where  ${P\in \CC(\bar K)}$ satisfies
\begin{align*}
t(P)&\in K\smallsetminus\mho,\\
\label{eteps}
|t(P)|_v&<\eps \qquad (v\in S),\\
%\label{eheb}
\Height(t(P))&\le B, 
\end{align*}
there are at least ${cB^\ell/\log B}$ distinct fields of degree~$d$ over~$K$. Moreover, if~$P$ is one of these points, then, by definition of our set~$\mho$, for any ${Q\in \psi^{-1}(P)}$ we have ${[K(Q):K(P)]=\deg(\psi)}$ and $[K(P):K]=\deg(t)$, which means that~$P$ satisfies 2). The theorem is proved. 
\end{proof}

%%%%%%%%%%%%%%%%%%%%%%%%%%%%%%%%%%%%%%%%%%%%%

%%%%%%%%%%%%%%%%%%%%%%%%%%%%%%%%%%%%%%%%%%%%%

\subsection{From finite subgroups of Jacobians to finite covers}

In this section we consider a commutative finite flat $\mathcal{O}_K$-group scheme $\G$, whose generic fiber $\G_K$ becomes constant cyclic of order $n$ over $\bar{K}$. Typical examples are $\Z/n\Z$ (the constant group scheme) and $\mu_n$ (the group scheme of $n$-th roots of unity). We denote by $\G^D$ the Cartier dual of $\G$.

Let $\CC$ be a curve over $K$, with a given $K$-rational point $A$, and let $J(\CC)$ be the Jacobian of $\CC$. The point $A$ gives rise to an embedding $i_A:\CC\to J(\CC)$ such that $i_A(A)=0$.

\begin{lemma}
\label{finite_covers}
Assume that $J(\CC)$ contains a subgroup scheme isomorphic to $(\G^D_K)^r$ for some integer $r\geq 1$. Then there exists $\G_K$-torsors
$$
\psi_1:X_1\to\CC,\dots,\psi_r:X_r\to \CC
$$
such that:
\begin{enumerate}
\item $\psi_i^{-1}(A)$ is the trivial torsor for all $i$;
\item the $\psi_i$ generate a subgroup of $H^1(\CC,\G_K)$ isomorphic to $(\Z/n\Z)^r$;
\item the fiber product $X_1\times_{\CC}\dots\times_{\CC} X_r$ is a geometrically irreducible curve.
\end{enumerate}
\end{lemma}

\begin{proof}
Let us fix an embedding $\G^D_K\to J(\CC)$. Let $B$ be the quotient abelian variety $J(\CC)/\G^D_K$, and let $\theta:J(\CC)\to B$ be the isogeny with kernel $\G^D_K$. This isogeny gives rise to a dual isogeny $\theta^t:B^t\rightarrow J(\CC)^t$ where $B^t$ is the dual abelian variety of $B$. By duality of abelian varieties, the kernel of $\theta^t$ is the Cartier dual of the kernel of $\theta$, hence is isomorphic to $\G_K$. By auto-duality of the Jacobian, we have $J(\CC)^t\simeq J(\CC)$. Let $\psi:X\to \CC$ be the morphism defined by the following cartesian square (or pull-back)
$$
\begin{CD}
X @>>> B^t \\
@V\psi VV @VV\theta^tV \\
\CC @>i_A>> J(\CC) \\
\end{CD}
$$
then $\psi$ is a $\G_K$-torsor, because $\theta^t$ is, and $\psi^{-1}(A)$ is the trivial torsor, because $(\theta^t)^{-1}(0)$ is.

Because $\CC(K)\neq \emptyset$, the Leray spectral sequence (see \cite[expos\'e V, \S{}3]{gro4t2}) associated to $\CC\to \Spec(K)$ and $\G_K$ gives us a short exact sequence
\begin{equation}
\label{spectral_tors}
0 \longrightarrow H^1(K,\G_K)  \longrightarrow H^1(\CC,\G_K)  \longrightarrow \Hom(\G_K^D,J(\CC)) \longrightarrow 0
\end{equation}
More precisely, the right hand side map above induces a bijection
$$
\ker\left(A^*:H^1(\CC,\G_K)\to H^1(K,\G_K)\right) \longrightarrow \Hom(\G_K^D,J(\CC))
$$
Indeed, according to the ker-coker Lemma, the section $A:\Spec(K)\to \CC$ gives rise to an isomorphism between the cokernel of the map $H^1(K,\G_K) \to H^1(\CC,\G_K)$ and the kernel of the map $A^*:H^1(\CC,\G_K)\to H^1(K,\G_K)$.

If we start from an injective morphism $\G_K^D\to J(\CC)$, we get from the construction above a geometrically irreducible $\G_K$-torsor $X\to \CC$ which belongs to $\ker(A^*)$. By hypothesis, we have $r$ independent injections $\G_K^D\to J(\CC)$, which generate a subgroup isomorphic to $(\Z/n\Z)^r$ in $\Hom(\G_K^D,J(\CC))$. Therefore, we obtain $r$ torsors $\psi_1,\dots \psi_r$ which generate a subgroup isomorphic to $(\Z/n\Z)^r$ in $\ker\left(A^*:H^1(\CC,\G_K)\to H^1(K,\G_K)\right)$. Their fiber product $X_1\times_{\CC}\dots\times_{\CC} X_r$ is a geometrically irreducible $(\G_K)^r$-torsor.
\end{proof}

\begin{theorem}
\label{main_application}
Let $\CC$ be a curve over $K$, and let $J(\CC)$ be the Jacobian of $\CC$. Let ${t:\CC\to\P^1}$ be a finite $K$-morphism  which is totally ramified at some $K$-rational point of $\CC$.% (this condition is satisfied as soon as $\CC(K)\neq\emptyset$).

Let $\G$ be a finite flat $\mathcal{O}_K$-group scheme, whose generic fiber $\G_K$ is isomorphic to the cyclic group $\Z/n\Z$ over $\bar{K}$, and let $r\geq 1$ be an integer such that $J(\CC)$ contains a subgroup scheme isomorphic to $(\G_K^D)^r$.

Then there exist an infinity of number fields $L/K$ with $[L:K]=\deg(t)$ such that the natural map
\begin{equation}
\label{einjmap}
H^1_{\fl}(\mathcal{O}_K,\G)\longrightarrow H^1_{\fl}(\mathcal{O}_L,\G)
\end{equation}
is injective, and satisfies
\begin{equation}
\label{esatisfiessomethingIdonotknowwhat}
\rank_n H^1_{\fl}(\mathcal{O}_L,\G)/H^1_{\fl}(\mathcal{O}_K,\G) \geq r.
\end{equation}
\end{theorem}

The ``infinity'' in this theorem can be made quantitative as follows. 

\begin{theorem}
\label{thquant_ma}
Let ${t\in K(\CC)}$ be the rational function defining  the $K$-morphism ${\CC\to \PP^1}$ appearing in Theorem~\ref{main_application}. 
Assume that there exists a rational function ${x\in K(\CC)}$ of degree~$m$ such that ${K(\CC)=K(t,x)}$. Then, for sufficiently large positive~$X$, there are at least $cX^{\ell/2m(d-1)}/\log X$
number  fields~$L/K$  with  ${[L:K]=d}$, ${\DD(L/K)\le X}$ and such that the map~\eqref{einjmap} is injective and satisfies~\eqref{esatisfiessomethingIdonotknowwhat}. Here ${\ell=[K:\Q]}$ and ${c>0}$ depends on~$\CC$,~$t$,~$x$ and~$K$. 
\end{theorem}

\begin{proof}[Proof of Theorem~\ref{main_application}]
The strategy of the proof is that of Theorem~2.4. of \cite{GL12}. According to Lemma~\ref{finite_covers}, we are able to choose $r$ torsors
$$
\psi_1:X_1\to\CC,\dots,\psi_r:X_r\to \CC
$$
such that $\psi_i^{-1}(A)$ is trivial for all $i$, and the fiber product $X_1\times_{\CC}\dots\times_{\CC} X_r$ is geometrically irreducible.

We apply Theorem~\ref{int_torsors} to $\Psi:X_1\times_{\CC}\dots\times_{\CC} X_r\to \CC$, which is a geometrically irreducible $\G_K^r$-torsor such that $\Psi^{-1}(A)$ is the trivial torsor. This proves the existence of infinitely many $P\in \CC(\bar{K})$ with $[K(P):K]=\deg(t)$ such that $\Psi^{-1}(P)$ is the spectrum of a field and extends into a flat $\G^r$-torsor over $\mathcal{O}_{K(P)}$. Then, for such points $P$, the torsors $\psi_1^{-1}(P),\dots, \psi_r^{-1}(P)$ can be extended into flat $\G$-torsors over $\mathcal{O}_{K(P)}$, and these torsors generate a subgroup isomorphic to $(\Z/n\Z)^r$ in $H^1(K(P),\G_K)$, hence:
$$
\rank_n H^1_{\fl}(\mathcal{O}_{K(P)},\G) \geq r.
$$

Furthermore, according to Theorem~\ref{int_torsors}, there are infinitely many isomorphism classes among the fields $K(P)$, and we may impose the additional requirement that $\Psi^{-1}(P)$ is the spectrum of a field which is linearly disjoint from a given finite extension $F/K$. Let us now choose $F/K$ to be the compositum of all the fields $K(\xi)$, where $\xi$ runs through $H^1_{\fl}(\mathcal{O}_K,\G)$; this is a finite extension, because the group $H^1_{\fl}(\mathcal{O}_K,\G)$ is finite, according to  Hermite-Minkowski's Theorem. We note that all points of $\G_K$ are defined over $F$.
The map $H^1_{\fl}(\mathcal{O}_K,\G)\to H^1_{\fl}(\mathcal{O}_{K(P)},\G)$ is injective because, $K(P)$ being linearly disjoint from $F$, a nontrivial $\G$-torsor over $\mathcal{O}_K$ cannot acquire a point over $K(P)$. Moreover, the image of $H^1_{\fl}(\mathcal{O}_K,\G)\to H^1_{\fl}(\mathcal{O}_{K(P)},\G)$ has trivial intersection with the subgroup generated by the $\psi_i^{-1}(P)$, because the compositum of the corresponding fields are linearly disjoint. Hence the result.
\end{proof}

\begin{remark}
\label{ramif_infinity}
In the statement of Theorem~\ref{int_torsors}, it is possible to enlarge the set $S$ as one likes. Therefore, Theorem~\ref{main_application} still holds when replacing the flat cohomology groups by the {\'e}tale ones, and more generally by the groups $H^1_{\text{$S$-split}}$ of $S$-split classes. In particular, one may impose the additional requirement that the torsors are unramified at infinite places.
\end{remark}

\begin{proof}[Proof of Theorem~\ref{thquant_ma}]
Our fields~$L$ occur as $K(P)$, where the points~$P$ are produced by a suitable special case of Theorem~\ref{int_torsors}. The ``quantitative statement'' of the latter theorem implies that, for large positive~$B$, there are at least $cB^\ell/\log B$ distinct fields among such $K(P)$ with ${\Height(P)\le B}$. Proposition~\ref{pdiscrim} implies now that for every such ${L=K(P)}$ we have ${\DD(L/K)\le c'B^{2m(d-1)}}$. Setting here ${B=X^{1/2m(d-1)}}$, we obtain the result. 
\end{proof}

%%%%%%%%%%%%%%%%%%%%%%%%%%%%%%%%%%%%%%%%%%%%%

%%%%%%%%%%%%%%%%%%%%%%%%%%%%%%%%%%%%%%%%%%%%%

\subsection{Proof of Theorems~\ref{thold},~\ref{thnew} and~\ref{thquant}}
\label{ssproofs}

If $A$ is a finite abelian group, we denote by $A^{\vee}:=\Hom(A,\Q/\Z)$ its Pontryagin dual. The group $A^{\vee}$ being (non canonically) isomorphic to $A$, we have, for any integer $n$,
$$
\rank_n A^{\vee}[n]=\rank_n A^{\vee}=\rank_n A.
$$

Let us recall the following result \cite[Lemma 2.6]{GL12}, that we shall extensively use below: if
$$
0 \longrightarrow A \longrightarrow B \longrightarrow C \longrightarrow 0
$$
is an exact sequence of $n$-torsion abelian groups, then we have
\begin{equation}
\label{gl12eqn}
\rank_n B \geq \rank_n A + \rank_n C.
\end{equation}
Moreover, if the exact sequence splits then the equality holds.

\begin{proof}[Proof of Theorem~\ref{thnew}]
Let $\G=\Z/n\Z$, then, according to class field theory, we have for any number field $L$ a canonical isomorphism
$$
H^1_{\text{$M_L^{\infty}$-split}}(\mathcal{O}_L,\Z/n\Z) \simeq \Hom(\Clg(L),\Z/n\Z)=\Clg(L)^{\vee}[n].
$$

Indeed, for an archimedean place, splitting and being unramified are the same. In fact, one could replace $M_L^{\infty}$ by the set of real archimedean places of $L$, since complex places never ramify.

Let $r=\rank_{\mu_n} J(\CC)$ be the maximal integer such that $J(\CC)$ has a subgroup scheme isomorphic to $\mu_n^r$. Then Theorem~\ref{main_application}, 
in the light of Remark~\ref{ramif_infinity}, implies that there exist infinitely many fields $L$ with $[L:K]=\deg(t)=d$ such that the map $\Clg(K)^{\vee}[n]\to \Clg(L)^{\vee}[n]$ is injective, and:
$$
\rank_n \Clg(L)^{\vee}[n]/\Clg(K)^{\vee}[n] \geq r.
$$
According to \eqref{gl12eqn}, this implies that
$$
\rank_n \Clg(L)^{\vee}[n] \geq r + \rank_n \Clg(K)^{\vee}[n],
$$
hence the result.
\end{proof}

Before we start the next proof, let us recall that Kummer theory (in flat topology) yields, for any number field $L$, an exact sequence
$$
0 \longrightarrow \mathcal{O}_L^{\times}/n \longrightarrow H^1_{\fl}(\mathcal{O}_L,\mu_n) \longrightarrow \Clg(L)[n] \longrightarrow 0.
$$
According to \cite[Prop.~1.1]{GG17}, this exact sequence always splits. As pointed above, it follows that the inequality \eqref{gl12eqn} is an equality:
\begin{equation}
\label{kummer_rank}
\rank_n H^1_{\fl}(\mathcal{O}_L,\mu_n) = \rank_n (\mathcal{O}_L^{\times}/n) + \rank_n \Clg(L)[n].
\end{equation}

\begin{proof}[Proof of Theorem~\ref{thold}]
Let $\G=\mu_n$, then we recover the situation considered in \cite{GL12}. We note that $\mu_n^D=\Z/n\Z$, hence if we let $r=\rank_n J(\CC)(K)_{\tors}$ then $J(\CC)$ contains a subgroup isomorphic to $(\mu_n^D)^r$.

In this setting, Theorem~\ref{main_application} reads as follows: there exist infinitely many fields $L$ with $[L:K]=\deg(t)$ such that $H^1(\mathcal{O}_K,\mu_n)\to H^1_{\fl}(\mathcal{O}_L,\mu_n)$ is injective and:
$$
\rank_n H^1_{\fl}(\mathcal{O}_L,\mu_n)/H^1_{\fl}(\mathcal{O}_K,\mu_n) \geq \rank_n J(\CC)(K)_{\tors}.
$$

According to \eqref{gl12eqn}, this implies that
$$
\rank_n H^1_{\fl}(\mathcal{O}_L,\mu_n) - \rank_n H^1_{\fl}(\mathcal{O}_K,\mu_n) \geq \rank_n J(\CC)(K)_{\tors}.
$$
According to \eqref{kummer_rank}, the left-hand side is equal to the following quantity:
$$
\rank_n (\mathcal{O}_L^{\times}/n) - \rank_n (\mathcal{O}_K^{\times}/n) + \rank_n \Clg(L) - \rank_n \Clg(K).
$$

Finally, let us point out that, by construction, the field $L$ is linearly disjoint from the field $K(\zeta_n)$. In other terms, $\mu_n(K)=\mu_n(L)$, which implies that
$$
\rank_n (\mathcal{O}_L^{\times}/n) - \rank_n (\mathcal{O}_K^{\times}/n) = \rank_{\Z} \mathcal{O}_L^\times - \rank_{\Z} \mathcal{O}_K^\times.
$$

Putting all the pieces together we obtain:
$$
\rank_{\Z} \mathcal{O}_L^\times - \rank_{\Z} \mathcal{O}_K^\times + \rank_n \Clg(L) - \rank_n \Clg(K) \geq \rank_n J(\CC)(K)_{\tors}.
$$
which yields the required lower bound on $\rank_n \Clg(L)$.
\end{proof}

\begin{proof}[Proof of Theorem~\ref{thquant}]
It is an immediate consequence of Theorem~\ref{thquant_ma}. 
\end{proof}

%%%%%%%%%%%%%%%%%%%%%%%%%%%%%%%%%%%%%%%%%%%%%

%%%%%%%%%%%%%%%%%%%%%%%%%%%%%%%%%%%%%%%%%%%%%

\section{Applications and examples}
\label{s_applications}

%%%%%%%%%%%%%%%%%%%%%%%%%%%%%%%%%%%%%%%%%%%%%

%\subsection{Short reminder on superelliptic curves}

In all our examples, $\CC$ is a superelliptic curve, defined (over $K$) by an affine equation of the form
$$
y^m=f(x),
$$
where $f$ is a separable polynomial with coefficients in~$K$. If the degree of $f$ is coprime to $m$, then the curve $C$ has a unique point $A_{\infty}$ at infinity, which is $K$-rational.
Moreover, there are two natural functions $C\to \PP^1$ whose unique zero is $A_{\infty}$, namely:
\begin{enumerate}
\item[(i)] the map $(x,y)\mapsto 1/x$, which has degree $m$;
\item[(ii)] the map $(x,y)\mapsto 1/y$, which has degree $\deg f$.
\end{enumerate}

Each of those is a natural candidate to play the role of the function $t$.

%%%%%%%%%%%%%%%%%%%%%%%%%%%%%%%%%%%%%%%%%%%%%

\subsection{Mestre's example revisited}

We briefly review the construction by Mestre \cite{Me92} from which the present paper is inspired. For the reader's convenience, we stick to Mestre's notation.

In his paper, Mestre constructs:
\begin{enumerate}
\item a genus $5$ hyperelliptic curve $\CC$ defined over $\Q$, which admits three rational Weierstrass points;
\item three elliptic curves $E_1$, $E_2$ and $E_3$ defined over $\Q$, each of them endowed with an isogeny $\varphi_i:E_i\to F_i$ with kernel $\Z/5\Z$;
\item three independent Galois covers $\tau_i:\CC\to F_i$ with group $(\Z/2\Z)^2$.
\end{enumerate}

The existence of the maps $\tau_i$ implies that the Jacobian of~$\CC$ splits, and that each of the $F_i$ is an isogenus factor of $J(\CC)$ via an isogeny of degree $4$. More precisely, there exists an abelian surface $B$ and an isogeny
$$
F_1\times F_2\times F_3 \times B \longrightarrow J(\CC)
$$
whose degree is a power of $2$.

On the other hand, the dual isogeny $\hat{\varphi_i}:F_i\to E_i$ has kernel $\mu_5$ (the Cartier dual of the constant group scheme $\Z/5\Z$). Hence $J(\CC)$ contains $\mu_5^3$ as a  subgroup.

Let us apply Theorem~\ref{thnew}  to this situation.

\begin{theorem}
Let $K$ be a number field. There exist infinitely many quadratic extensions $L/K$ such that
$$
\rank_5 \Clg(L) \geq 3 + \rank_5 \Clg(K).
$$
More precisely, for every large positive~$X$ there exist at least $cX^{\ell/22}$ such fields~$L$ with ${\DD(L/K)\le X}$. Here ${\ell=[K:\Q]}$ and~$c$ is an absolute positive constant. 
\end{theorem}

When ${K=\Q}$, we recover Mestre's result.

%%%%%%%%%%%%%%%%%%%%%%%%%%%%%%%%%%%%%%%%%%%%%

\subsection{Extensions of cyclotomic fields: proof of Theorem~\ref{thp}}

Let us recall the following result of Greenberg \cite[Theorem~1]{greenberg81}, obtained in his work on Jacobians of quotients of Fermat curves.

\begin{theorem}
\label{Greenberg}
Let $p\geq 3$ be a prime, and let $s$ be an integer such that $1 \leq s\leq p-2$. Let $\Q(\zeta_p)$ be the $p$-th cyclotomic field, and let $C_{p,s}$ be the curve defined by the affine equation
$$
y^p=x^s(1-x).
$$
Then $J(C_{p,s})(\Q(\zeta_p))$ contains a subgroup isomorphic to $(\Z/p\Z)^3$.
\end{theorem}

Theorem~\ref{thp} is a simple consequence of this result and Theorem~\ref{thnew}.

\begin{proof}[Proof of Theorem~\ref{thp}]
Let us put ${s:=d-1}$ and let $C_{p,s}$ be the curve defined in Theorem~\ref{Greenberg}. Then the rational map $t:C_{p,s}\to \PP^1$ defined by $t:=1/y$ has degree $d$. Because $d$ is coprime to $p$, its unique zero is the point at infinity.

On the other hand, if ${\zeta_p\in K}$ then ${\mu_p\simeq \Z/p\Z}$ over~$K$, hence over that field we have
$$
\rank_{\mu_p} J(C_{p,s}) = \rank_p J(C_{p,s})(K). 
$$
According to Theorem~\ref{Greenberg}, the right-hand side is at least $3$. Therefore, the result follows from Theorem~\ref{thnew}.
\end{proof}

In case ${d\geq 5}$, one can improve on Theorem~\ref{thp} as follows. 

\begin{theorem}
\label{thyetanother}
Let~$p$ be a prime number,~$K$ a number field containing~$\zeta_p$, and  let~$d$ be coprime to~$p$. Then there exist infinitely many fields $L/K$ with ${[L:K]=d}$ such that
$$
\rank_p \Clg(L) \geq d-1 + \rank_p \Clg(K).
$$
\end{theorem}

\begin{proof}
Let us consider the curve $\CC$ defined by an equation of the form
$$
y^p=(x-a_1)\dots (x-a_d)
$$
where the $a_i$ are pairwise distinct rational numbers. Then $J(\CC)(\Q)$ contains a subgroup isomorphic to $(\Z/p\Z)^{d-1}$, which becomes isomorphic to $(\mu_p)^{d-1}$ over $K$. The result now follows from Theorem~\ref{thnew}.
\end{proof}

We omit the ``quantitative'' version, which can be done in the same way as previously. 

\begin{remark}
By considering the same curve over~$\Q$, and applying Theorem~\ref{thold} (instead of Theorem~\ref{thnew}), one recovers the following result, due to Azuhata and Ichimura \cite{AI84}: there exist infinitely many $L/\Q$ with $[L:\Q]=d$ such that
${\rank_p \Clg(L) \geq \left\lfloor {d}/{2} \right\rfloor}$. 
In this result, the base field is $\Q$ instead of $\Q(\zeta_p)$, but the $p$-rank is half the degree, whereas in Theorem~\ref{thyetanother} the $p$-rank has the size of the degree.
\end{remark}

%%%%%%%%%%%%%%%%%%%%%%%%%%%%%%%%%%%%%%%%%%%%%

%%%%%%%%%%%%%%%%%%%%%%%%%%%%%%%%%%%%%%%%%%%%%

\begin{small}

\end{small}

%%%%%%%%%%%%%%%%%%%%%%%%%%%%%%%%%%%%%%%%%%%%%

%%%%%%%%%%%%%%%%%%%%%%%%%%%%%%%%%%%%%%%%%%%%%

\vskip 30pt

Yuri Bilu
\smallskip

Institut de Math{\'e}matiques de Bordeaux

351, cours de la Lib{\'e}ration

33405 Talence Cedex, France
\smallskip

\texttt{yuri.bilu@math.u-bordeaux.fr}

\vskip 30pt

Jean Gillibert
\smallskip

Institut de Math{\'e}matiques de Toulouse

118, route de Narbonne

31062 Toulouse Cedex 9, France
\smallskip

\texttt{jean.gillibert@math.univ-toulouse.fr}

%%%%%%%%%%%%%%%%%%%%%%%%%%%%%%%%%%%%%%%%%%%%%

\end{document}